\documentclass[11pt]{amsart}
\usepackage{amssymb,amsthm,amsmath}
\usepackage{enumerate}
 
\newtheoremstyle{hessu}                
{3pt}{3pt}{}{\parindent}{\bfseries}{.}{.5em}{}

\newtheorem{teoreema}{Theorem}[section]
\newtheorem{lemma}{Lemma}[section]
\newtheorem{Remark}{Remark}[section]
\newtheorem{definition}{Definition}[section]

\def\oN{\mathbb N}
\def\oQ{\mathbb Q}
\def\oR{\mathbb R}
\def\oC{\mathbb C}
\def\oZ{\mathbb Z}
\def\ma{\mathfrak{A}}
\def\mb{\mathfrak{B}}
\def\soc{second order characterizable}
\def\sock{$L^2_{\kappa\omega}$-characterizable}

\title{On Second-order Characterizability}
\author{T. Hyttinen, K. Kangas and J. V\"a\"an\"anen}
\thanks{Research of the first and third authors was partially supported by
grant 40734 of the Academy of Finland}
\thanks{Research of the second author was partially supported by
the Wihuri Foundation} 
\thanks{Research of the third author was partially supported by
 the European Science Foundation EUROCORES LogICCC programme LINT project.}

\begin{document}
\maketitle

\begin{abstract} We investigate the extent of \soc\ structures by extending Shelah's {\em Main Gap} dichotomy to second order logic. For this end we consider a countable complete first order theory $T$.  We show that all sufficiently large models of $T$ have a characterization up to isomorphism in the extension of second order logic obtained by adding a little bit of infinitary logic if and only if $T$ is shallow superstable  with NDOP and NOTOP. Our result relies on cardinal arithmetic assumptions. Under weaker assumptions we get  consistency results or alternatively results about second order logic with Henkin semantics.
\end{abstract}


\tableofcontents

\section{Introduction}
 
Let us call a structure $\ma$ {\em second order characterizable} if there is a second order sentence $\phi$ in the vocabulary of $\ma$ such that $$\mb\models\phi\iff \ma\cong\mb$$ for all structures $\mb$. If $\phi$ was required to be first order, very few structures would be characterizable in this way, in fact only the finite ones. But with a second order $\phi$ almost any familiar structure occurring in mathematics is characterizable. The most obvious examples are $(\oN,<), (\oZ,+,\otimes,0,1),$ $(\oQ,+,\times,0,1), (\oR,+,\times,0,1)$ and $(\oC,+,\times,0,1)$. What is the extent of \soc\ structures in mathematics? In this paper we attempt to answer this question.

There are obvious cardinality restrictions to the extent to which arbitrary structures can be \soc\ in that there are only countably many second order sentences to begin with. We overcome this restriction by focusing on the infinitary logic $L^2_{\kappa\omega}$, the extension of $L_{\kappa\omega}$ by second order quantification. Even so, the number of possible characterizable structures is limited, and characterizing every structure is out of the question. It seems reasonable to think that structures that are in some vague sense ``natural" are \soc\ but structures that are in some vague sense ``arbitrary", e.g. constructed with the help of the Axiom of Choice, can be proved not to be  \soc, and indeed our results support this claim.

The Axiom of Choice is often used in mathematics to construct objects which have pathological properties and which seem to elude explicit definition. Typical examples are the construction of a well-ordering of the reals and constructions of less obvious bases for vector spaces. It is well-known that if $V=L$, then there is a well-ordering $\prec$ of $\oR$ such that $(\oR,+,\times,0,1,\prec)$ is \soc, but if a Cohen real is added (or if certain large cardinals exist) the existence of such a well-order $\prec$ is impossible. In Section~\ref{Hamel} we show that the same is true of Hamel bases for the reals. 

Ajtai~\cite{MR548473} proved that if $V=L$, then the second order theory of a countable structure determines the structure up to isomorphism, and on the other hand, if a Cohen real is added, then there are countable structures that are not determined up to isomorphism by their second order theory. Keskinen~\cite{Kes} extended this to uncountable structures by replacing second order logic by $L_{\kappa\omega}^2$. Solovay~\cite{solovay} focused on finitely axiomatizable second order theories and showed with an argument similar to Ajtai's that if $V=L$, then every finitely axiomatized complete second order theory is categorical, but consistently there are finitely axiomatized complete second order theories that are non-categorical. We complement the results of Ajtai and Keskinen by investigating how general these results are. We take an arbitrary countable complete first order theory $T$ and an uncountable cardinal $\kappa$ and ask, does $T$ have  models of cardinality $\kappa$ that are not second order (or rather $L^1_{\kappa\omega}$-) characterizable? It turns out that the answer to this question follows the dividing line of Shelah's Main Gap Dichotomy: Such models always exist, subject to assumptions about cardinal arithmetic, if and only if $T$ is shallow superstable and without DOP or OTOP (Theorem~\ref{main}).

What do our results say about second order theories? Suppose $T$ is a complete second order (more exactly $L^2_{\kappa\omega}$) theory in a countable vocabulary. What can be said about the categoricity of $T$? Of course, we assume that $T$ is consistent, i.e. has a model $\ma$ of some cardinality $\kappa$. If $\kappa$ is big enough (as in the assumptions of Theorem~\ref{main}), we can make the following conclusions about categoricity  in cardinality $\kappa$: Let $T_0$ be the set of first order consequences of $T$. Of course, $T_0$ is a complete theory. If $T_0$ is shallow superstable and without DOP or OTOP, every model of $T_0$, and hence of $T$, of cardinality $\kappa$ is \soc\ in the sense of $L^2_{\kappa\omega}$. If $T$ is finitely axiomatizable, then moreover $T$ is categorical, as a finitely axiomatized complete second order theory can have only models of one cardinality. On the other hand, suppose $T_0$ is unsuperstable or superstable with DOP or OTOP. Then $T_0$ has two models $\ma$ and $\mb$ which are $L^2_{\kappa\omega}$-equivalent but non-isomorphic. Thus, if $T^*$ is the  $L^2_{\kappa\omega}$-theory of $\ma$, then $T^*$ is a complete $L^2_{\kappa\omega}$-theory which is non-categorical and which has the same first order consequences as our original $T$. So our model theoretic criterion does not decide the categoricity of $T$ but only the categoricity of a theory which has the same models as $T$ up to first order elementary equivalence.  

One way to describe our results is the following: Suppose we are interested in finding structures of cardinality $\kappa$ that are not $L^2_{\kappa\omega}$-characterizable. We may be interested in groups, fields, linear orders, Boolean algebras, graphs, equivalence relations, or combinations of those. All we need to do is to choose a structure of the appropriate kind the first order theory of which is unsuperstable or superstable with DOP or OTOP. Then we get from our general results a non-$L^2_{\kappa\omega}$-characterizable model of this theory.

On the other hand, suppose we want to use second order logic to analyze models of a shallow superstable first order theory without DOP or OTOP. From our general results we know that for $\kappa$ satisfying our particular assumptions every model of cardinality $\kappa$ of the theory is $L^2_{\kappa\omega}$-characterizable. So there is an $L^2_{\kappa\omega}$-sentence which acts as a perfect invariant for the model. By studying  the proof we see where second order logic is used and we can actually analyze the invariant further.

In Section~\ref{Henkin} we study second order characterizability in the more general framework of normal models, that is, general models  in the sense of Henkin~\cite{MR12:70b} satisfying the Comprehension axioms of Hilbert and Ackermann \cite{MR0351742}. The more general framework permits us to get results without cardinal arithmetic assumptions. The lesson then is, that non-categoricity of second order theories is a common phenomenon but if we want to manifest non-categoricity by means of {\em full} models (models in which the second order variables range over {\em all} subsets and relations of the domain), we have to make cardinal arithmetic assumptions or use forcing.

{\bf Notation:} $S^\kappa_\lambda$ is the set of ordinals $<\kappa$ of cofinality $\lambda$. $\mathcal{P}(A)$ denotes the power-set of the set $A$. $A\Delta B$ denotes the symmetric difference of $A$ and $B$, i.e. $(A\setminus B)\cup(B\setminus A)$. CH denotes the Continuum Hypothesis. AC denotes the Axiom of Choice. CA denotes the Comprehension Axioms of second (and higher) order logic introduced by Hilbert and Ackermann \cite{MR0351742}.

\section{Second order characterizability of Hamel bases}\label{Hamel}

The classical structures on which mathematics is largely based, such as  $(\oN,<), (\oZ,+,\otimes,0,1),$ $(\oQ,+,\times,0,1), (\oR,+,\times,0,1)$ and $(\oC,+,\times,0,1)$, are easily seen to be \soc. The question can be asked, what are the structures arising in  mathematics that are {\em not} \soc? In this section we give one example, the Hamel bases for the reals as a $\oQ$-vector space. 

It was proved in \cite{MR763890} that the Axiom of Choice is equivalent to the statement that every vector space has a basis. Thus the question of the existence of a (Hamel) basis for the reals as a $\oQ$-vector space is  related to AC. 
Another equivalent of AC is the Well-Ordering Principle: Every set can be well-ordered. If $V=L$, then there is a $\Sigma^1_2$-well-ordering of the reals, a result that goes back to G\"odel. It was proved in  \cite{MR0176925} that if a Cohen real is added by forcing, there is no well-ordering of the reals that would be definable in set theory, hence in that case there  is a fortiori also no well-ordering $\prec$ of the reals such that $(\oR,+,\times,0,1,\prec)$ would be \soc. It there are infinitely many Woodin cardinals, then all projective sets of reals are Lebesgue measurable (\cite{MR959110,MR959109}) and hence there can be no such well-ordering $\prec$.

\begin{teoreema}
If we add one Cohen real, then in the extension there is no $B\subseteq\oR$ such that the following hold:
\begin{enumerate}
\item $B$ is a Hamel basis for $\oR$, i.e. a basis for $\oR$ as a $\mathbb{Q}$-vector space.
\item  The structure $(\oR, +, \times, 0, 1, B)$ is second-order characterizable.
\end{enumerate}  On the other hand, if $V=L$, then such a $B$ does exist.
\end{teoreema}

\begin{proof}
We force with the partial order $\mathbb{P} = \{p: n \rightarrow \{0,1\} \textrm{  } | \textrm{  } n \in \omega \}$, ordered by inclusion. Let $G$ be a $\mathbb{P}$-generic filter over the ground model $V$.  We define the real $r^G$ so that its binary expansion is 
\begin{equation*}
r^G=0.{r_0}^G{r_1}^G{r_2}^G...
\end{equation*}
where
\begin{displaymath}
{r_i}^G=(\bigcup G)(i).
\end{displaymath}

Let $B \subseteq \mathbb{R}$ be a basis  for $\mathbb{R}$ as a vector space over $\mathbb{Q}$. We make a counterassumption that there is a second-order formula $\phi$ characterizing the structure $(\mathbb{R},+,\times,0,1,B)$. Now we can find a basis $B' \subseteq \mathbb{R}$ such that  $1 \in B'$, and a second-order formula $\psi$ for which
\begin{displaymath}
x \in B' \Leftrightarrow (\mathbb{R}, +, \times, 0, 1) \models \psi(x).
\end{displaymath}
This is done as follows. Let  $a_0, \ldots, a_n \in B$ be such that 
\begin{displaymath}
1=\lambda_0 a_0 + \ldots + \lambda_n a_n 
\end{displaymath}
for some $\lambda_0, \ldots, \lambda_n \in \mathbb{Q}$. Now choose $b_1, \ldots, b_n$ so that $\textrm{span}(1,b_1, \ldots, b_n)=\textrm{span}(a_0, a_1, \ldots, a_n)$. Let $\lambda_0^i, \ldots, \lambda_n^i \in \mathbb{Q}$ be such that
\begin{displaymath}
 b_i=\lambda_0^i a_0 + \ldots + \lambda_n^i a_n, 
\end{displaymath}  
for $1 \leq i \leq n$. Now we can write
\begin{eqnarray*}
\psi(x)  =  \exists B_0 (\phi(+, \times, 0, 1, B_0)  \wedge \exists a_0 \ldots  \exists a_n ( \bigwedge_{0\leq i \leq n}B_0 a_i   \wedge 1=\lambda_0 a_0 + \ldots + \lambda_n a_n \\ \wedge ((B_0 x \wedge \bigwedge_{0 \leq i \leq n} \neg x=a_i) \lor x=1 \lor \bigvee_{1 \leq i \leq n} x=\lambda_0^i a_0 + \ldots + \lambda_n^i a_n))).
\end{eqnarray*}
Since every rational number is expressible in our language, we may think of the numbers $\lambda_j, \lambda_j^i$, for $0 \leq j \leq n$ and $1 \leq i \leq n$, as parameters in the formula.  
  
There are $q_0, q_1, \ldots, q_n \in \mathbb{Q}$, and $b_1, \ldots, b_n \in B'$ such that  
\begin{eqnarray*}
r^G=q_0 1 + q_1 b_1 + \ldots + q_n b_n,
\end{eqnarray*}
$q_i \neq 0$ for $0<i\leq n$, $b_i \neq b_j$ for $j \neq i$, and $b_i \neq 1$ for $0<i \leq n$.
Thus, there is a condition $p \in G$ such that
\begin{displaymath}
\begin{array}{ll}
p \Vdash & r^G = \check{q_0} \check{1} + \check{q_1} \dot{b_1} + \ldots + \check{q_n} \dot{b_n} 
\ \wedge \\
& \bigwedge_{1 \leq i \leq n} \neg \check{q_i}=0 
 \wedge \bigwedge_{1 \leq i < j \leq n} \neg \dot{b_i}=\dot{b_j} 
\  \wedge\\
&   \bigwedge_{1 \leq i \leq n} \psi(\dot{b_i})\,  \wedge \, \psi(\check{1}) \,\wedge \bigwedge_{1 \leq i \leq n} \neg \dot{b_i}= \check{1} .
\end{array}
\end{displaymath}
Let $n=\textrm{dom } p$. Choose $m>n$ so that $\bigcup G(m)=0$. Then the function $f: \, \mathbb{P} \rightarrow \mathbb{P}$,
\begin{displaymath}
f(p)(i)=\left\{ \begin{array}{ll}
p(i) & \textrm{if $i \neq m$}\\ 1-p(i) & \textrm{if $i=m$}\\  
\end{array} \right.
\end{displaymath}
is an automorphism of the partial order $\mathbb{P}$, and $f \in V$. Let
\begin{displaymath}
H = \{f(q) \, | \, q \in G \}.
\end{displaymath}
Then, $H$ is $\mathbb{P}$-generic over $V$, and it is easy to see that $H \in V[G]$ and $G \in V[H]$. Thus, $V[H]=V[G]$. 

As $m \notin \textrm{dom }p$, we have $f(p)=p$, and  $p \in H$. Thus, we can write
\begin{displaymath}
r^G = {q_0} 1 +  q_1 {b_1} ^G + \ldots + q_n {b_n}^G
\end{displaymath}
and
\begin{displaymath}
r^H = {q_0} 1 +  q_1 {b_1} ^H + \ldots + q_n {b_n}^H, 
\end{displaymath}
where $b_i^G$, $b_i^H$ $\in B'$. Let
\begin{displaymath}
q=r^H-r^G.
\end{displaymath}
We have changed only one number in the binary expansion, so we get
\begin{displaymath}
q=0.0\ldots01 \in \mathbb{Q}.
\end{displaymath} 
Thus,
\begin{displaymath}
r^H=r^G+q=({q_0}+q)1 +  q_1 {b_1} ^G + \ldots + q_n {b_n}^G.
\end{displaymath}
As the representation of a vector with respect to the basis is unique, we have $q=0$. Thus $r^G=r^H$ which is a contradiction. 

The claim concerning $V=L$ follows from the existence, assuming $V=L$, of a second order definable well-order of the reals.

\end{proof}

Second order characterizability can be generalized in at least two ways: we can allow  third or even higher order quantifiers, or we can allow infinite conjunctions and disjunctions. Most of our results extend from second order to higher order. In Section~\ref{Henkin} we deal explicitly with higher order logic. The introduction of infinitary logic in connection with second order characterizability is particularly relevant because of the following result of Scott \cite{MR34:32}: Every countable structure in a countable vocabulary is characterizable in $L_{\omega_1\omega}$ up to isomorphism among countable structures. If we try to extend this to uncountable structures we find that several obvious routes are blocked. Uncountable $L_{\infty\omega}$-equivalent structures need not be isomorphic even if they are of the same cardinality. For example, $\aleph_1$-like dense linear orders without a first element are all $L_{\infty\omega_1}$-equivalent, but need not be isomorphic \cite{MR0462942}. We introduce now a useful amalgam of second order and infinitary logic. This concept was introduced in \cite[Chapter 5]{Kes}:
  
\begin{definition}
Let $\kappa$ be a regular cardinal. The logic $L^2_{\kappa \omega}$ is defined as follows:
\begin{itemize}
\item atomic sencond-order formulas are $L^2_{\kappa \omega}$-formulas,
\item if $\phi$ is an $L^2_{\kappa \omega}$-formula, then $\neg \phi$, $\exists x \phi$, $\forall x \phi$, $\exists X \phi$, and $\forall X \phi$ are $L^2_{\kappa \omega}$-formulas.
\item if $\phi_i$, $i<\lambda<\kappa$, are $L^2_{\kappa \omega}$-formulas and only finitely many individual or second order variables occur in $\{\phi_i:i<\lambda\}$, then $\bigwedge_{i<\lambda} \phi_i$ and $\bigvee_{i<\lambda} \phi_i$ are $L^2_{\kappa \omega}$-formulas.
\end{itemize}
\end{definition}

Since we will be dealing with concepts of characterization applied to several different logics, we define:

 \begin{definition}
 We say that a structure $\ma$ is \emph{characterizable in the logic $L$} if there exists an $L$-sentence $\phi$ such that $\ma\models\phi$ and any structure $\mb$, for which $\mb \models \phi$, is isomorphic with $\ma$.
 \end{definition}

We can now immediately observe that every model $\ma$ of size $\kappa$ is characterizable in $L_{\kappa^+ \omega}^{2}$:
Let $L$ be the vocabulary of $\ma$ and let $R$ be a new binary predicate. 
We expand $\ma$ to a an $L \cup \{R\}$-model $\ma^{*}$ by interpreting $R$ so that it well-orders $\ma$ with the order type $\kappa$.
Now it is easy to find an $L_{\kappa^+ \omega}$-sentence $\phi$ such that $\ma^{*} \models \phi$ and for all $L \cup \{R\}$-structures $\mb$, if the interpretation of $R$ well-orders $\mb$ with the order-type $\kappa$ and $\mb \models \phi$, then $\mb \cong \ma^{*}$. 
 Also, it is easy to find a $L_{\kappa^+ \omega}$-sentence $\psi$ such that for all $L \cup \{R\}$-models $\mb$, $\mb \models \phi$ if and only if the interpretation of $R$ well-orders $\mb$ with the order type $\kappa$. 
Then the formula $\exists R(\psi \wedge \phi)$ characterizes $\ma$.

Consequently, the interesting question is, whether a structure of cardinality $\kappa$ is $L^2_{\kappa\omega}$-characterizable. Generalizing the work of Ajtai \cite{MR548473} Keskinen proved

\begin{teoreema}[\cite{Kes}]\label{keskinen}If $V=L$, then $L^2_{\kappa\omega}$-equivalence implies isomorphism among structures of cardinality $\kappa$ of a finite vocabulary. On the other hand, there is for every infinite $\kappa$ a forcing extension which preserves cardinals such that in the extension there are two $L^2_{\kappa\omega}$-equivalent non-isomorphic models of cardinality $\kappa$. 
\end{teoreema}
  
\begin{teoreema}
Assume CH. If we add $\omega_1$ many Cohen reals, then in the extension there is no $B\subseteq\oR$ such that the following hold:
\begin{enumerate}
\item $B$ is a Hamel basis for $\oR$, i.e. a basis for $\oR$ as a $\mathbb{Q}$-vector space.
\item  The structure $(\oR, +, \times, 0, 1, B)$ is characterizable in $L^{2}_{\omega_1 \omega}$.
\end{enumerate}  On the other hand, if $V=L$, then such a $B$ does exist.
\end{teoreema}
 
\begin{proof}
We force using the partial order   
\begin{displaymath}
\mathbb{P}=\{p:  \omega_1 \times \omega \rightarrow \{0,1\} \, | \, |p|<\omega \},
\end{displaymath}
ordered by inclusion. For each $\alpha < \omega_1$, let $\mathbb{P}_{\alpha} = \{p \in \mathbb{P} \, | \, \textrm{dom}(p) \subseteq \alpha \}$ and $\mathbb{P}^{\alpha} = \{p \in \mathbb{P} \, | \, \textrm{dom}(p) \cap \alpha=\emptyset \}$. Now, for each $\alpha < \omega_1$, $\mathbb{P} \cong \mathbb{P}_{\alpha} \times \mathbb{P}^{\alpha}$, and if $G$ is $\mathbb{P}$-generic over $V$, then $G \cap \mathbb{P}_{\alpha}$ is $\mathbb{P}_{\alpha}$-generic over $V$, $G \cap \mathbb{P}^{\alpha}$ is $\mathbb{P}^{\alpha}$-generic over $V[G \cap \mathbb{P}_{\alpha}]$, and $V[G \cap \mathbb{P}_{\alpha}][G \cap \mathbb{P}^{\alpha}] = V[G]$.

We make a counter-assumption that there is a  $L^{2}_{\omega_1 \omega}$ -formula $\phi$ characterizing structures that satisfy the conditions (1)-(3). We can view the formula $\phi$ as a tree with no infinite branches where each node is labelled with some atomic or negated atomic formula or with one of the symbols $\vee$, $\wedge$, $\exists v_i$, $\forall v_i$, $\exists X_i$, or $\forall X_i$, for $i \in \omega$. Thus, we can code the formula $\phi$ as a subset of some countable set $S \in V$. Hence, $\phi$ has a nice name 
\begin{displaymath}
\dot{\phi} = \bigcup \{\{\check{s}\} \times A_s \, | \, s \in S\},
\end{displaymath}
where each $A_s$ is an antichain in $\mathbb{P}$.
As $\mathbb{P}$ has the c.c.c. and $S$ is countable in $V$, $\dot{\phi}$ is countable, and thus it is a $\mathbb{P}_{\alpha}$-name for some $\alpha < \omega_1$. Therefore $\phi \in V[G\cap \mathbb{P}_{\alpha}]$ for some $\alpha < \omega_1$.  

We define the real $r^G$ so that its binary expansion is 
\begin{equation*}
r^G=0.{r_0}^G{r_1}^G{r_2}^G...
\end{equation*}
where
\begin{displaymath}
{r_i}^G=(\bigcup G)(\alpha,i),
\end{displaymath}
and proceed as in the proof of the previous theorem to get a contradiction.

The claim concerning $V=L$ follows, as in the previous theorem, from the existence, assuming $V=L$, of a second order definable well-order of the reals.
\end{proof}

\section{First order theories and second order logic}\label{fot}

Above we found that consistently there is no \sock\ expansion of the field of real numbers by a predicate for a Hamel basis. We aim now at a  general result which would tell us what kind of non-\sock\ structures we can find and, respectively, what kind of structures of cardinality $\kappa$ are {\em a fortiori} \sock. 

We take the approach of stability theory, more exactly that of {\em classification theory} \cite{Sh}. For the concepts of shallowness, stability, superstability, DOP and OTOP we refer to standard texts in stability theory, e.g. \cite{Sh}. 

Suppose $T$ is a countable complete first order theory. In a major result called {\em Main Gap Theorem} Shelah proved that if $T$ is shallow, superstable without DOP and OTOP, then in uncountable cardinalities $\kappa$ the models of $T$ can be described up to isomorphism in terms of dimension-like invariants; in particular, $L_{\infty\kappa}$-equivalence implies isomorphism among models of $T$ of cardinality $\kappa$. On the other hand, if $T$ is not shallow, superstable without DOP and OTOP, then in each uncountable cardinality $\kappa$ the theory $T$ has non-isomorphic models that are highly equivalent, in particular $L_{\infty\kappa}$-equivalent.

An intuitive description of the Main Gap Theorem is that by simply looking at the stability-theoretic properties of a first order theory $T$ we can decide whether its models can be analyzed in terms of geometric and algebraic concepts, or whether all of it also permits  models that are so complicated that no geometric or algebraic analysis will reveal their isomorphism type. In the latter case further work (\cite{HytTu,FrHyKu}) has extended the original result of Shelah and has yielded more and more complicated models.

Our goal  here is to show that models of a shallow superstable first order theory without DOP or OTOP are \sock, and in the opposite case the theory has models that are non \sock. We can fulfill this only partially.

We start with a result which is illuminating even if not the most general:

\begin{teoreema}\label{v2te}
Assume CH. If we add a Cohen subset to $\omega_1$, then in the forcing extension the following holds: If $T$ is a countable complete unstable theory then it has models $\ma$ and $\mb$ of size $\aleph_1$ such that $\ma$ and $\mb$ are $L_{\omega_1 \omega}^{2}$-equivalent but non-isomorphic.
\end{teoreema}

Note that the model $\ma$ in the above theorem is necessarily non-$L^2_{\omega_1\omega}$-characterizable, because it cannot be distinguished from $\mb$ even with an $L^2_{\omega_1\omega}$-theory and not even just among models of cardinality $\aleph_1$. Theorem~\ref{keskinen} shows that we cannot hope to get the above result provably in ZFC. However, in a simple forcing extension we get non-$L^2_{\omega_1\omega}$-characterizable models of cardinality $\aleph_1$ for any unstable theory. It is worth noting that completeness and unstability of a first order theory are absolute properties in set theory.
   
For the proof of the above theorem we first sketch the proof of a crucial lemma. The lemma and thus also Theorem \ref{v2te} hold also for theories with OTOP and for unsuperstable theories. For more details we refer to \cite{FrHyKu}.

\begin{lemma}\label{lemma1}
Assume CH. Let $T$ be a countable unstable theory. For each $S \in \mathcal{P}(S_\omega^{\omega_1})$ we may define a model $\ma(S)  \models T$ of size $\omega_1$ such that for $S, S' \in \mathcal{P}(S_\omega^{\omega_1})$, $\ma(S) \cong \ma(S')$ if and only if $S \bigtriangleup S'$ is nonstationary. Moreover, there is a a template $\Phi_T$ (as in \cite{Sh}, Section V2. Lemma 2.4), and a first order formula (in set theory) $\phi(y,z, \omega_1, \Phi_T)$ such that for all $S \subseteq S_\omega^{\omega_1}$ and $\ma$, $\phi(S, \ma, \omega_1, \Phi_T)$ holds if and only if $\ma \cong \ma(S)$.
 \end{lemma}

\begin{proof}
(Sketch)
For each set $S \subseteq S_\omega^{\omega_1}$ we define the linear order $\Phi(S)$ as follows:
\begin{displaymath}
\Phi(S)=\sum_{\alpha < \omega_1} \eta_\alpha,
\end{displaymath}
where
\begin{displaymath}
\eta_\alpha= \left\{ \begin{array}{ll}
\mathbb{Q} & \textrm{if $\alpha \notin S$}\\1+ \mathbb{Q}  & \textrm{if $\alpha \in S$}, 
\end{array} \right.
\end{displaymath}
where $\mathbb{Q}$ denotes the usual ordering of rational numbers. 

Denote by ${Tr}^\omega$ the set of structures $\ma =(T, <, \ll, (T_n)_{n\le \omega}, h)$, where $T=\{f: \, n \rightarrow \eta \, | \, n \le \omega \}$ for some linear order $\eta$, and the other symbols are interpreted as follows:
\begin{itemize}
\item $f<g \iff f \subseteq g$,
\item $f \ll g$ $\iff$ $f<g$ or there is some $n \in \textrm{dom } f \cap \textrm{dom } g$ such that $f \upharpoonright n = g \upharpoonright n$ and $f(n)<_\eta g(n)$, where $<_\eta$ denotes the ordering relation of $\eta$,
\item $h(f,g)$ is the maximal common initial segment of $f$ and $g$, and
\item $T_n=\{f \in T \, | \, \textrm{dom }f=n \}$.
\end{itemize}
For each $S \subseteq S_\omega^{\omega_1}$ we define the tree $T(S) \in {Tr}^\omega$ by
\begin{displaymath}
\begin{array}{lcl}
T(S)=\Phi(S)^{<\omega} &\cup &\{f: \omega \rightarrow \Phi(S)  
f \textrm{ is increasing and }\\
&&\textrm{has a least upper bound} \},
\end{array}\end{displaymath}
with the relations $<$, $\ll$, $h$ and $T_n$ interpreted in the natural way.

Denote by $L$ the vocabulary of $T$. Let $T^{*}$ be the Skolemisation of $T$ and let $L^{*}$ be the vocabulary of $T^{*}$. Now, similarly as in proof of Theorem 80, Claim 4, in \cite{FrHyKu}, we can define for every tree $T(S)$ a suitable linear order $L(T(S))$ and denote by $\ma^{1}(S)$ the Ehrenfeucht-Mostowski model $EM^1(L(T(S)), \Phi_T)$, where $\Phi_T$ is the template as in \cite{Sh}, Section V2, Lemma 2.4. 
As in \cite{FrHyKu}, we may define inside $\ma^{1}(S)$ the $T(S)$-skeleton of $\ma^{1}(S)$, $\{a_f \, | \, f \in T(S)\}$, so that
\begin{itemize}
\item There is a mapping $T(S) \rightarrow (\textrm{dom } \ma^{1}(S))^n$ for some $n \in \omega$, $f \mapsto a_f$, such that $\ma^{1}(S))^n= \textrm{SH}(\{a_f \, | \, f \in T(S)\})$.
\item  $\ma(S)=\ma^{1}(S)\upharpoonright L$ is a model of $T$.
\item $\{a_f \, | \, f \in T(S)\}$ is indiscernible in $\ma^{1}(S))$, i.e. if $\bar{f}, \bar{g} \in T(S)$ and $\textrm{tp}_\textrm{q.f.} (\bar{f}/\emptyset)=\textrm{tp}_\textrm{q.f.} (\bar{g}/\emptyset)$, then $\textrm{tp}(a_{\bar{f}} /\emptyset)=\textrm{tp}(a_{\bar{g}} /\emptyset)$. This assignment of types is independent of $S$ and depends only on the template $\Phi_T$.
\end{itemize}
Clearly the size of $\ma(S)$ is $2^\omega = \omega_1$. Let $\phi(x,y,\Phi_T)$ be the first-order formula  defining the connection between the sets  $S \subseteq S_\omega^{\omega_1}$ and the models $\ma(S)$. 

Suppose now $S, S' \subseteq S_\omega^{\omega_1}$ and $S \bigtriangleup S'$ is nonstationary. We show that $\Phi(S) \cong \Phi(S')$, and thus $T(S) \cong T(S')$, whence $\ma(S) \cong \ma(S')$. Let $C$ be a cub set such that $C \cap (S \bigtriangleup S')=\emptyset$. Enumerate it by $C=\{\alpha_i \, | \, i<\omega_1\}$, where $(\alpha_i)_{i<\omega_1}$ is an increasing sequence containing all its limit points. Denote for each $S \subseteq S_\omega^{\omega_1}$, and $\alpha< \beta < \omega_1$,
\begin{displaymath}
\Phi(S, \alpha, \beta) =\sum_{\alpha \leq i < \beta} \eta_i.
\end{displaymath}
Now we can write $$\Phi(S)=\bigcup_{i<\omega_1} \Phi(S, \alpha_i, \alpha_{i+1})$$ and $$\Phi(S')=\bigcup_{i<\omega_1} \Phi(S', \alpha_i, \alpha_{i+1}).$$ These are disjoint unions, so it suffices to show that for all $i < \omega_1$, the orders $\Phi(S, \alpha_i, \alpha_{i+1})$ and $\Phi(S', \alpha_i, \alpha_{i+1})$ are isomorphic. Since we chose $C$ so that it does not intersect  $S \bigtriangleup S'$, we have for all $i < \omega_1$,
\begin{displaymath}
\alpha_i \in S \Leftrightarrow \alpha_i \in S'.
\end{displaymath}
Thus, for each $i < \omega_1$, we have either
\begin{displaymath}
 \Phi(S, \alpha_i, \alpha_i+1) \cong \mathbb{Q} \cong  \Phi(S', \alpha_i, \alpha_i+1),
 \end{displaymath}
 if $\alpha_i \notin S$, or
\begin{displaymath}
 \Phi(S, \alpha_i, \alpha_i+1) \cong 1+\mathbb{Q} \cong  \Phi(S', \alpha_i, \alpha_i+1),
 \end{displaymath}
 if $\alpha_i \in S$.

It can also be shown that if $S \bigtriangleup S'$ is stationary, then $\ma(S) \ncong \ma(S')$. This is done as in the proof of Theorem 80, Claim 5 in \cite{FrHyKu}.
\end{proof}

Now we can prove Theorem~\ref{v2te}:
 
\begin{proof}
We force using the partial order 
\begin{displaymath}
\mathbb{P}=\{f: \alpha \rightarrow \{0,1\} \, | \, \alpha < \omega_1\},
\end{displaymath}
ordered by inclusion. This partial order is $\omega_1$-closed and has the $\omega_2$-c.c. Thus the forcing preserves cardinals and does not add subsets to $\omega$. In particular, CH holds in the forcing extension and every countable theory is in the ground model $V$.
 
 It follows from Lemma \ref{lemma1} that there is a first-order formula $\phi(y,z, \omega_1, \Phi_T)$, where $\Phi_T \in V$, such that for each $S \in \mathcal{P}(S_\omega^{\omega_1})$ there is some model $\ma \models T$ of size $\omega_1$ for which $\phi(S, \ma, \omega_1, \Phi_T)$ holds. Moreover, for all $S, S'  \in \mathcal{P}(S_\omega^{\omega_1})$ and $\ma, \mb \models T$, if $\phi(S, \ma, \Phi_T) \wedge \phi(S', \mb, \Phi_T)$ holds, then $\ma \cong \mb$ if and only if $S \bigtriangleup S'$ is nonstationary.  
  
Let $G$ be a $\mathbb{P}$-generic filter over $V$ and denote
\begin{displaymath}
S_G=(\bigcup G)^{-1} (1)\cap S_\omega^{\omega_1}.
\end{displaymath}
We choose models $\ma, \mb \models T$ of size $\omega_1$ such that $\phi(S_G, \ma, \Phi_T)$ and $\phi(S_\omega^{\omega_1} \setminus S_G, \mb, \Phi_T)$.  
 As $S_G \bigtriangleup (S_\omega^{\omega_1} \setminus S_G)=S_\omega^{\omega_1}$ is stationary, $\ma \ncong \mb$.
 
We show that $\ma$ and $\mb$ are $L_{\omega_1 \omega}^{2}$-equivalent. 
Let $\psi \in L_{\omega_1 \omega}^{2}$, and suppose $\ma \models \psi$. Now $\psi \in V$.
As the formula $\phi$ defines the isomorphism type, there is a condition $p \in G$ such that
 \begin{displaymath}
 p \Vdash \forall \mathcal{C}(\phi(\dot{S_G}, \mathcal{C}, \check{\Phi_T}) \rightarrow \mathcal{C} \models \check{\psi}).
 \end{displaymath}
 Let $\textrm{dom }p=\gamma$, and denote
\begin{displaymath}
G^*_\gamma=\{f^*_\gamma \, | \, f \in G \},
\end{displaymath}
where 
\begin{displaymath}
f^*_\gamma(\alpha)=\left\{ \begin{array}{ll}
f(\alpha) & \textrm{if $\alpha < \gamma$}\\ 1-f(\alpha) & \textrm{otherwise.}\\  
\end{array} \right.
\end{displaymath}
The function $f \mapsto f^*_\gamma$ is an automorphism of $\mathbb{P}$ and thus $G^*_\gamma$ is $\mathbb{P}$-generic over $V$ and $V[G]=V[G^*_\gamma]$. Also, 
\begin{displaymath}
S_{G^*_\gamma}=(S_G \cap \gamma) \cup (S_\omega^{\omega_1} \setminus (S_G \cup \gamma)).
\end{displaymath}
Let $\mathcal{C}$ be such that $\mathcal{C} \models T$, $|\mathcal{C}| = \omega_1$, and  $\phi(S_{G^*_\gamma}, \mathcal{C}, \Phi_T)$ holds in $V[G]$. Now, $p \in G^*_\gamma$, and thus $\mathcal{C} \models \psi$. The set $X=(S_\omega^{\omega_1} \setminus S_G) \bigtriangleup S_{G^*_\gamma}$ is nonstationary as $X \subseteq \gamma$. Therefore, $\mathcal{C} \cong \mb$, and hence $\mb \models \psi$.
 \end{proof}

Above $\aleph_1$ we get stronger results:

\begin{teoreema}\label{pakotus}
Assume that  $\kappa$ and $\lambda$ are cardinals such that $\kappa = \lambda^+=2^\lambda$, $\lambda^{<\lambda}=\lambda>\omega$. 
Then in the forcing extension that we get by adding a Cohen subset to $\kappa$, the following holds: If $T$ is a countable complete theory that is either unsuperstable or superstable with OTOP, then it has two non-isomorphic $L_{\kappa \omega}^{2}$-equivalent models of size $\kappa$.
If we assume further that  $\lambda > 2^\omega$, then the claim holds also for superstable theories with DOP.
\end{teoreema}

\begin{proof}
This can be proved as Theorem \ref{v2te}. The fact that there is a suitable first-order formula  $\phi(y,z, \kappa, \Phi_T)$ follows from Theorems 80 (in the case that $T$ is unstable or superstable with DOP or OTOP) and 87 (in the case that $T$ is stable but not superstable) in \cite{FrHyKu}. 
\end{proof}

For even bigger cardinals we get a complete characterization:

\begin{teoreema}\label{main}
Suppose that $\kappa$ is a regular cardinal such that $\kappa=\aleph_\alpha$, $\beth_{\omega_1}(|\alpha|+\omega) \leq \kappa$ and $2^\lambda<2^\kappa$ for all $\lambda < \kappa$. Let $T$ be a countable complete first order theory. Then every model of $T$ of size $\kappa$ is \sock\  if and only if $T$ is a shallow, superstable theory without DOP or OTOP.
\end{teoreema}  

\begin{proof}
Suppose first $T$ is either an unsuperstable theory or a superstable theory that is deep or has DOP or OTOP. Then by Shelah's Main Gap Theorem (see \cite{Sh}, Section X2.6) $T$ has $2^\kappa$ nonisomorphic models of size $\kappa$. 
However, every formula in $L_{\kappa \omega}^{2}$ can be coded as a subset of $\lambda$ for some $\lambda < \kappa$. 
Thus, there are at most $2^{<\kappa}$ many formulas in $L_{\kappa \omega}^{2}$. 
Since $\textrm{cf } 2^\kappa>\kappa$, and by our assumptions $2^\lambda<2^\kappa$ for every $\lambda < \kappa$, we have  $2^{<\kappa}<2^\kappa$.
Hence, by the pigeonhole principle, the models of $T$ of size $\kappa$ cannot be characterized in $L_{\kappa \omega}^{2}$.

On the other hand, suppose $T$ is a shallow, superstable theory without DOP or OTOP, and let $\mathcal{M}$ be a model of $T$ of size $\kappa$. We show that  $\mathcal{M}$ is characterizable in $L_{\kappa \omega}^{2}$.

As in Definiton 3.1. in \cite{Sh}, Section XI, we mean by $A \subseteq^a B$ that for every $\bar{a} \in A$ and $\bar{b} \in B$ there is $\bar{b}' \in A$ such that
\begin{displaymath}
\textrm{stp(}\bar{b}/\bar{a}\textrm{)}=\textrm{stp(}\bar{b}'/\bar{a}\textrm{)}.
\end{displaymath}

By Claim 2.6 in \cite{Sh}, section XI, there is a $(\mathbf{T^{t}_{\aleph_0}}, \subseteq^a)$ -decomposition of $\mathcal{M}$, i.e. a triple $(P, f, g)$ satisfying the following conditions:
\begin{enumerate}[(a)]
\item $P=(P, \ll)$ is a tree, $f: P \setminus \{r\} \rightarrow \bigcup_{n \in \omega}M^n$, where $r$ is the root of $P$, and $g: P \rightarrow \mathcal{P}(M)$ are functions,  
\item if $t,u,v \in P$ are such that $t=u^{-}$ and $u=v^{-}$, then $\textrm{tp}(f(v)/g(u))$ is orthogonal to $g(t)$,
\item if $t,u,v \in P$, $t=u^{-}$, and $u \ll v$ or $u=v$, then $\textrm{tp}_{*}(g(v)/g(t) \cup f(u))$ is almost orthogonal to $g(t)$,
\item if $t \in P$, then $\{f(u) \, | \, t=u^{-}\}$ is a maximal set such that it is  independent over $g(t)$ and satisfies (b),
\item if $t,u \in P$, and $t=u^{-}$, then $\textrm{tp}(f(u)/g(t))$ is regular,
\item if $t , u \in P$, $t=u^{-}$, $g(u) \subseteq A \subseteq M$ and $\textrm{tp}_*(A/ g(t) \cup f(u))$ is almost orthogonal to $g(t)$, then $g(u) \subseteq^a A$,
\item $g(r) \subseteq^a M$.
\end{enumerate}  
(See also Definitions 2.4 and 2.5.  and Context 2.1 in \cite{Sh}, Section XI.)
Notice that as $T$ is superstable, the models $g(t)$ may be chosen so that $|g(t)| \leq 2^\omega$ for each $t \in P$.

By Theorem 2.8 and Conclusion 3.17 in \cite{Sh}, Section XI, the model $\mathcal{M}$ is prime over $\bigcup_{t\in P} g(t)$ (see also Definition 2.2). 
Thus all models of $T$ with the same decomposition (up to isomorphism)  are isomorphic to $\mathcal{M}$. 
We will show that it is possible to write a $L_{\kappa \omega}^{2}$-sentence $\phi$ such that if $\mathcal{N}$ is an $L$-structure, then $\mathcal{N} \models \phi$ if and only if there is a decomposition of $\mathcal{N}$ isomorphic to $(P, f, g)$. 
 Then $\phi$ characterizes structures isomorphic to $\mathcal{M}$.  
          
Note that we can code all ordinals and cardinals less than $\kappa$ as elements of $M$. 
We first introduce a new binary predicate $<$ and say that it is a well-order of $M$ of order type $\kappa$ (this can be easily done in second-order language). Now, for each ordinal $\alpha < \kappa$, we can define a formula $\psi_\alpha$ such that
\begin{displaymath}
\psi_\alpha (x) \iff x \in \kappa \textrm{ and } x \geq \alpha.
\end{displaymath}
This is done by induction. If $\alpha=0$, then
\begin{displaymath}
\psi_\alpha (x) = "x=x",
\end{displaymath}
if $\alpha=\beta+1$ for some ordinal $\beta < \kappa$, then
\begin{displaymath}
\psi_\alpha (x) = \exists y (y < x \wedge \psi_{\beta}(y)),
\end{displaymath}
and if $\alpha$ is a limit ordinal, then
\begin{displaymath}
\psi_\alpha (x) = \bigwedge_{\beta < \alpha}\psi_\beta.
\end{displaymath}
Using these formulas, we can express "$x=\alpha$" for any ordinal $\alpha < \kappa$ as "$\psi_\alpha(x) \wedge \neg \psi_{\alpha+1}(x)$". 
This allows us to use ordinals  $\alpha < \kappa$ as parameters in our formulas.
For instance, for $\alpha_0, \ldots, \alpha_n < \kappa$,  and a formula $\phi(x_0, \ldots, x_n)$, we can code $\phi(\alpha_0, \ldots, \alpha_n)$ as $\forall x_0 \ldots \forall x_n(\bigwedge_{i \leq n} \psi_{\alpha_i}(x_i) \rightarrow \phi(x_0, \ldots, x_n))$.

It is also easy to define a predicate $P_{card}$ such that $P_{card}(x)$ holds if and only if $x$ codes some cardinal $\lambda < \kappa$. Let $\alpha$ be the ordinal such that $\kappa=\aleph_\alpha$. It follows from our assumptions that $\alpha < \kappa$. We can now code (and use as parameters) all cardinals $\aleph_\beta$ for $\beta < \alpha$ in a similar fashion as the ordinals. This will allow us to describe the cardinalities of various sets.
 
Since $T$ is shallow, the tree $P$ has no infinite branches, and since $T$ is superstable, we can choose the models $g(t)$ so that $|g(t)| \leq 2^\omega$ for all $t \in P$.
To be able to speak of elements of  the set $\bigcup_{t \in P} g(t)$, we enumerate this set using a function 
\begin{displaymath}
h: P \times 2^\omega \cdot \omega \longrightarrow \bigcup_{t \in P} g(t),
\end{displaymath}
 with the following properties. 
For each $t \in P$, 
\begin{displaymath}
g(t)=\{h(t,i) \, | \, i< 2^\omega \cdot (\textrm{ht}(t)+1)\}, 
\end{displaymath}
and if $u \in P$ is such that $u \ll t$, then $h(u,i)=h(t,i)$ for each $i < 2^\omega \cdot \textrm{ht}(t)$.  As $|M|= \kappa > 2^\omega \cdot \omega$, we can find a set $U \subseteq M$ such that $(U,<) \cong 2^\omega \cdot \omega$ and use the set $P \times U$ to enumerate $\bigcup_{t \in P} g(t)$. Thus our sentence $\phi$ will begin as follows.
$$\begin{array}{ll}
\phi=\exists P  \exists U \exists < \exists \ll \exists f \exists h (&``P \cap U = \emptyset " \wedge\\
& ``(P, \ll) \textrm{ is a tree with no infinite branches}" \wedge\\ 
&``< \textrm{ is a well-ordering}" \wedge\\
& ``(U,<) \cong 2^\omega \times \omega" \wedge \ldots).
\end{array}$$
Here $f$ denotes the function from the definition of the decomposition of a model, and $h: P \times U \rightarrow M$ is the function used to enumerate $\bigcup_{t \in P} g(t)$ as described above. All the properties written in shorthand can be easily expressed in second-order language.  
      
As $2^\omega$ is small compared to $\kappa$, we can express isomorphism types of models of  cardinality $2^\omega$. This will allow us to describe the properties listed in the definition of decomposition. 
There are at most $2^{2^\omega}$ nonisomorphic models of $T$ of size $2^\omega$.
For each $n < \omega$, let $\{\mathcal{M}^n_i \, | \, i<2^{2^\omega} \}$ enumerate models of $T$ that have domain $2^\omega \cdot n$.
Suppose now $t \in P$ and $n=\textrm{ht}(t)+1$. 
For each $i<2^{2^\omega}$ we can write a formula $\phi_i^n(t)$ expressing that the mapping $H: h(t, j) \mapsto j$ is an isomorphism from $g(t)$ to $\mathcal{M}^n_i$.
   
Now we can express the properties needed when describing the decomposition. For instance, in case (b), we can use isomorphism types to describe orthogonality as follows. 
Let $t,u,v \in P$ be such that $t=u^{-}$ and $u=v^{-}$.
Let $n=\textrm{ht}(u)$.
As before, let $\{\mathcal{M}_i^{n+2} \, | \, i<2^{2^\omega} \}$ enumerate all models of $T$ with domain $2^\omega \cdot (n+2)$.
For each $i$, let $p^i_j \in S^{\mathcal{M}_i^{n+2}}(2^\omega \cdot (n+1))$ enumerate the types orthogonal to $2^\omega \cdot n$. 
As $T$ is superstable, we have $j<2^\omega$.
Now "$\textrm{tp}(f(v)/g(u))$ is orthogonal to $g(t)$" can be expressed using the formula
\begin{displaymath}
\bigwedge_{i<2^{2^\omega}}\phi_i^{n+2}(t) \rightarrow (\bigvee_{j<2^\omega} "\textrm{tp}(f(v)/g(u))=H^{-1}(p^i_j)").
\end{displaymath}
Note that at least one of the formulas $\phi_i^{n+2}(t)$ must be true as we are listing all isomorphism types. The concepts of almost orthogonal and regular can be described in a similar fashion.  

For (d) we need to be able to express that for each $t \in P$ the set $S=\{f(u) \, | \, t=u^{-}\}$ is independent over $g(t)$. This can be done, as we know that the set $S$ is independent over $g(t)$ if and only if every finite tuple from $S$ is independent over $g(t)$.
Let $n=\textrm{ht}(t)+1$. 
For $i<2^{2^\omega}$, $j<2^\omega$, let $p^i_j \in S^{\mathcal{M}_i^{n+1}}(2^\omega \cdot n)$ enumerate the types of finite tuples that are independent over $2^\omega \cdot n$.
 Now we can express `$S$ is independent over $g(t)$" with the formula
\begin{displaymath}
\begin{array}{ll}
\bigwedge_{i<2^{2^\omega}}(\phi_i^{n+1} \rightarrow \bigwedge_{m < \omega} \forall u_0 ... \forall u_m (&({u_0}^{-}=t \wedge ... \wedge {u_m}^-=t) \rightarrow \\
&\bigvee_{j<2^\omega}  p^i_j(f(u_0), ..., f(u_m)))). 
\end{array}
\end{displaymath} 
In second-order language it is easy to express also that $S$ is maximal with respect to (b).
 
For (f) and (g) we need to be able to express $g(t) \subseteq^a A$ for $t \in P$ and $A \subseteq M$. 
Let $n=\textrm{ht}(t)+1$.
For each $\bar{a} \in g(t)$, let $\psi_{i,j}^{\bar{a}}$, $i<2^{2^\omega}$, $j<\omega$, enumerate all formulas such that $\psi_{i,j}^{\bar{a}}(\bar{x},\bar{y}, H(\bar{a}))$ defines a finite equivalence relation in $\mathcal{M}_i^n$. Using the enumeration for $g(t)$ we can write a formula equivalent to
\begin{displaymath}
\bigwedge_{i<2^{2^\omega}} (\phi_i^n \rightarrow \bigwedge_{\bar{a} \in g(t)} (\forall \bar{b} \in A \exists \bar{c} \in g(t) (\bigwedge_{i<\omega}\psi_{i,j}^{\bar{a}}(\bar{c},\bar{b},\bar{a})))).
\end{displaymath}
This expresses $g(t) \subseteq^a A$. As there are only $2^\omega$ many finite tuples in $g(t)$, the first conjunction does not cause any problems. 
 
Finally, we have to describe the skeleton of $\mathcal{M}$ up to isomorphism.
With each node $t \in P$ we associate a formula $\theta_t (x)$ such that $\mathcal{M} \models \theta_t(t)$.  
We will define  $\theta_t(x)$ by induction from top down. First we need the notion of the rank of $t$, $\textrm{rk}(t)$. It is defined as follows.  
\begin{displaymath}
\textrm{rk}(t)=\left\{ \begin{array}{ll}
0 & \textrm{if $t$ is a leaf of $P$,}
\\ \textrm{sup} \{\textrm{rk}(u)+1 \, |\, t \ll u \} & \textrm{otherwise.}\\  
\end{array} \right.
\end{displaymath}
Now we can define the formula $\theta_t(x)$ by induction on $\textrm{rk}(t)$. If $\textrm{rk}(t)=0$, we let
$$\begin{array}{l}
\theta_t(x)=P(x) \wedge \\
``\neg \exists y (\textrm{tp}(y/h(\{x\} \times 2^\omega \cdot (\textrm{ht}(t)+1))) \perp h(\{x\} \times 2^\omega \cdot \textrm{ht}(t)))" 
\wedge \\
\bigwedge \textrm{tp} (h(\{t\} \times 2^\omega \cdot (\textrm{ht}(t)+1))/\emptyset),
\end{array}$$
where the symbol $\perp$ denotes orthogonality.
    
Suppose now we have defined $\theta_u(x)$ for all $u \in P$ such that $\textrm{rk}(u)<\textrm{rk}(t)$. Let
\begin{displaymath}
X_t=\{\theta_u \, | \, u \in P, \textrm{ rk}(u)<\textrm{rk}(t)\},
\end{displaymath}
and for each $u$, let
\begin{displaymath}
\lambda_u = |\{y \in P \, | \, t=y^{-} \textrm{ and } \mathcal{M} \models \theta_u (y) \}|.
\end{displaymath}
The idea is to write a formula $\theta_t(x)$ expressing that for each $u$, the node $x$ has $\lambda_u$ many direct successors $y$ such that $\theta_u(y)$ holds, and describing the types of $g(t)$ and $f(t)$. Thus
\begin{eqnarray*}
\theta_t(x)= P(x) \wedge \bigwedge_{\theta_u \in X_t} \exists P_u(\forall y(P_u(y) \leftrightarrow (x=y^{-} \wedge \theta_u(y))) \wedge |P_u|=\lambda_u) \\
\wedge  \bigwedge \textrm{tp} (h(\{t\} \times 2^\omega \cdot (\textrm{ht}(t)+1))/\emptyset).
\end{eqnarray*}
To be able to take conjunctions of the formulas in $X_t$ we have to prove that $|X_t|<\kappa$ for every $t \in P$. We do this by induction on $\textrm{rk}(t)$.
 
If $\textrm{rk}(t)=0$, then $t$ is a leaf, so $X_t=\emptyset$. 
Suppose now that $\textrm{rk}(t)=i>0$. 
For each $j<i$, denote $X_j=\{\theta_u \, | \, u \in P, \textrm{ rk}(u)< j\}$, and let $\lambda_j$ be a cardinal such that $|X_j| \leq \lambda_j$.
Let $\alpha$ be the ordinal such that $\kappa=\aleph_\alpha$. 
Assume first that $i=j+1$ for some ordinal $j$.
Let $v \in P$ be such that $\textrm{rk}(v)=j$.
For each formula $\psi \in X_j$, the number of nodes $w \in P$ such that $w^{-}=v$ and $\theta_w=\psi$, is either some $n \in \omega$ or $\aleph_\beta$ for some $\beta \leq \alpha$.
Moreover, there are $2^{2^\omega}$ possibilities for the type
\begin{displaymath}
\textrm{tp} (h(\{v\} \times 2^\omega \cdot (\textrm{ht}(v)+1))/\emptyset). 
\end{displaymath}
Thus, 
\begin{displaymath}
|X_t| \leq 2^{2^\omega} \cdot (\omega+|\alpha|)^{\lambda_j}.
\end{displaymath}
Suppose now that $i$ is a limit ordinal. Then $X_t=\bigcup_{j<i} X_j$, and we get
\begin{displaymath}
|X_t| \leq 2^{2^\omega} \cdot (\omega+|\alpha|)^{\textrm{sup}_{j<i}\lambda_j}.
\end{displaymath}

Hence we get the following bounds for $|X_t|$. If $\textrm{rank}(t)=0$, then $|X_t| < 1$.
If $\textrm{rank}(t)=1$, then
 \begin{displaymath}
 |X_t| \leq 2^{2^\omega}\cdot (\omega+|\alpha|) \leq \beth_{2}(\omega+|\alpha|).
 \end{displaymath}
In general, if $t \in P$ and $\textrm{rank}(t)=\beta$, then 
\begin{displaymath}
|X_t| \leq \beth_{\beta+1}(\omega+|\alpha|).
\end{displaymath}
Let $r$ be the root of $P$.
As $T$ is a shallow theory, $\textrm{rk}(r)<\omega_1$. 
Thus for every $t \in P$, $|X_t|< \beth_{\omega_1}(|\alpha|+\omega) \leq \kappa$.
 
The sentence
\begin{displaymath}
\exists x ((\theta_r(x)) \wedge \neg \exists y (P(y) \wedge y \ll x))
\end{displaymath}
describes the skeleton up to isomorphism. 
As we are able to express all the properties from the definition of decomposition, we may now compose a sentence $\phi$   characterizing all models of $T$ isomorphic to $\mathcal{M}$.
 \end{proof}

\section{Second order characterizability in models of CA}\label{Henkin}

Monadic second order logic over a structure $(A,R_0,...,R_n)$ can be thought of as first order logic over the enhanced structure \begin{equation}\label{sort1}
(A\cup P_1,A,P_1,\in,R_0,...,R_n),
\end{equation}
where $P_1=\mathcal{P}(A)$ and $\in$ is restricted to $A\times\mathcal{P}(A)$. Respectively, monadic third order logic can be related to first order logic of 
\begin{equation}\label{sort2}(A\cup P_1\cup P_2,A,P_1,P_2\in,R_0,...,R_n),\end{equation}
where $P_2=\mathcal{P}(P_1)$. For non-monadic higher order logics similar translations exist. One can reduce the entire type theory to first order logic in this way. The price one pays is that structures are limited to the very special form of (\ref{sort1}) and (\ref{sort2}). Henkin \cite{MR12:70b} took the natural step of considering the following more general structures than (\ref{sort1}):\begin{equation}\label{sort3}
(A\cup P_1,A,P_1,E,R_0,...,R_n),
\end{equation}
where $E$ is just a binary predicate $\subseteq A\times P_1$ satisfying the Extensionality Axiom. In addition, the Comprehension Axioms of \cite{MR0351742} are assumed. The Comprehension Axioms say that any definable relation on $A$ is canonically represented by an element of $P_1$. Henkin  \cite{MR12:70b} proved that such models yield a Completeness Theorem for second (and higher) order logic with respect to the obvious rules of inference that were introduced in \cite{MR0351742}. The original model (\ref{sort1}), called the {\em full} model, is of course a special case of (\ref{sort3}) and satisfies the Comprehension Axioms. Thus results about the models (\ref{sort1}) can be considered generalizations of results about the models (\ref{sort3}). When we prove existence results this generality means that our results are weaker than corresponding results about full models. However, our results use respectively weaker assumptions. 

Note that the Comprehension Axioms create unstability in models of the form (\ref{sort3}) even if there are no relations $R_i$ at all. Namely, by means of the Comprehension Axioms one can code finite sequences and manifest the so-called {\em independence property} (\cite{Sh}) of stability theory, a well-known special case of unstability. However, the components of the structure (\ref{sort3}) are not equal: the components $A,R_0,...,R_n$ constitute the underlying mathematical structure, while $P_1$ and $E$ play the auxiliary role of indicating what the range of the second order variables is. Thus it makes sense to ask whether the first order theory of $(A,R_0,...,R_n)$, rather than that of $(A\cup P_1,A,P_1,E,R_0,...,R_n)$, is for example stable or unstable and then analyze what we can say about the second order part encoded by $P_1$ and $E$; e.g. is $(A,R_0,...,R_n)$ second order characterizable in the version of second order logic encoded by $P_1$ and $E$.

We shall now introduce the main concept of this section, the concept of a {\em $\beta$-order-model}. This is the model (\ref{sort3}) taken to higher orders and made more exact. We have a slight bias in favor of finite order logics in our results, as logics of order $\omega$ and higher are conceptually  more complex.
 
 \begin{definition}\label{betamod}
Let $\beta$ be a countable ordinal, and let $L^{*}=\{R_0, \ldots, R_n \}$ be a relational vocabulary. 
Let $L=L^{*} \cup \{P_\beta,<,V, \in \}\cup\{\underline{\alpha}\ | \alpha < \beta\}$, where $P_\beta$ is a unary predicate, $<$, $V$, and $\in$ are binary predicates, and for each ordinal $\alpha<\beta$, $\underline{\alpha}$ is a constant symbol.
Let $\mathcal{W}$ be an $L$-structure such that
\begin{itemize}
\item $V^\mathcal{W} \subseteq {P_\beta}^\mathcal{W} \times (W \setminus {P_\beta}^\mathcal{W})$,
\item $<^\mathcal{W} \subseteq ({P_\beta}^\mathcal{W})^2$, and $({P_\beta}^\mathcal{W}, <^\mathcal{W}) \cong (\beta, \in)$,
\item $\in^\mathcal{W} \subseteq (W \setminus {P_\beta})^2$,
\item for each ordinal $\alpha < \beta$, $\underline{\alpha}^\mathcal{W} \in P_\beta^\mathcal{W}$, and $(\{x \in P_\beta^\mathcal{W} \, | \, x<^\mathcal{W} \underline{\alpha}\}, <^\mathcal{W}) \cong (\alpha, <)$. 
\end{itemize}
For each $\alpha < \beta$, we denote $V_\alpha=\{x \in W \setminus {P_\beta} \, | \, (\underline{\alpha}, x) \in V \}$. 
For an element $x \in W \setminus {P_\beta}$, we write $x \subseteq^\mathcal{W} V_\alpha$, if $y \in V_\alpha$ for every $y \in W \setminus {P_\beta}$ such that $y \in^\mathcal{W} x$.
 
The structure $\mathcal{W}$ is called a \emph{$\beta$-order-model} if in addition it satisfies the following conditions.
\begin{itemize}
\item If $x \in W\setminus (P_\beta^\mathcal{W} \cup V_0)$, then there is an ordinal $\alpha$ such that $\alpha+1<\beta$,  $x \in V_{\alpha+1}\setminus V_\alpha$, and $x \subseteq^{\mathcal{W}} V_\alpha$, 
\item if $\alpha < \gamma < \beta$, then $V_\alpha \subseteq V_\gamma$,  
\item if $\alpha<\beta$ is a limit ordinal, then $V_\alpha = \bigcup_{\gamma < \alpha}V_\gamma$, 
\item $W \setminus {P_\beta}= \bigcup_{\alpha < \beta} V_\alpha$,
\item for $0 \leq i \leq n$, $R_i^\mathcal{W} \subseteq V_0^{n}$, where $n$ is the arity of $R_i$.
\end{itemize}
 
For a $\beta$-order-model $\mathcal{W}$, and $\alpha<\beta$, we denote by $V_\alpha^\mathcal{W}$ the set $V_\alpha$ in that particular model. 
\end{definition}

In a $\beta$-order-model we may express quantifications of order $\alpha$ for any ordinal $\alpha <\beta$ by using what we call bounded quantifiers:

\begin{definition}
Let $L$ be the language of a $\beta$-order-model for some countable ordinal $\beta$, and let $\phi$ be an $L$-formula.
We say that a quantifier in the formula $\phi$ is \emph{bounded} if it is either of the form $\forall x \in V_{\alpha+1} \setminus V_\alpha$ or of the form $\exists x \in V_{\alpha+1} \setminus V_\alpha$ for some variable $x$ and some ordinal $\alpha+1 < \beta$.
We say a formula $\phi$ is \emph{bounded} if it contains only bounded quantifiers.
\end{definition}

\begin{definition}
Let $\ma$ and $\mb$ be $\beta$-order-models for some countable ordinal $\beta$.
We write $\ma \equiv^B \mb$ if the models satisfy the same bounded sentences.
Moreover, we write $\ma \preccurlyeq^B \mb$, if $\ma \subseteq \mb$, and
\begin{displaymath}
\ma \models \phi(\bar{a}) \iff \mb \models \phi(\bar{a})
\end{displaymath}
whenever $\bar{a} \in \ma$ and $\phi$ is a bounded formula.
\end{definition}

\begin{Remark}\label{standardbm}
\begin{enumerate}[(i)]
\item Let $n<\omega$. For every formula $\phi(\bar{x})$ we can find a bounded formula $\psi(\bar{x})$   such that  for every $n$-order-model $\mathcal{W}$ it holds that $\mathcal{W} \models \forall \bar{x}(\phi(\bar{x}) \leftrightarrow \psi(\bar{x}))$.
\item Let $L^{*}=\{R_0, \ldots, R_n \}$ be a relational vocabulary, and let $\mathcal{M}$ be an $L^{*}$-model. 
Let $\beta$ be a countable ordinal.
We define $S_\beta(\mathcal{M})$, \emph{the full $\beta$-order-model over $\mathcal{M}$}, to be the $\beta$-order-model such that
\begin{itemize}
\item $V_0^{S_\beta(\mathcal{M})}=\textrm{dom}(\mathcal{M})$,
\item for every ordinal $\alpha+1 < \beta$, $V_{\alpha+1}^{S_\beta(\mathcal{M})}=\mathcal{P}(V_\alpha^{S_\beta(\mathcal{M})}) \cup V_\alpha^{S_\beta(\mathcal{M})}$,
\item for every limit ordinal $\gamma < \beta$, $V_\gamma^{S_\beta(\mathcal{M})}=\bigcup_{\alpha<\gamma}V_\alpha^{S_\beta(\mathcal{M})}$,
\item $\in$ is interpreted in the natural way.
\end{itemize}
\item Let $n<\omega$, and let $\phi$ be a formula of order $n$. 
There is a bounded first-order formula $\psi$ such that for every $\beta>n$ and every $\mathcal{M} \models \phi$, $S_\beta(\mathcal{M}) \models \psi$.
\item Let $\mathcal{W}$ be a $\beta$-order-model for some countable ordinal $\beta$, and suppose there is a unique element $x \in V_1^\beta \setminus V_0^\beta$ satisfying the formula $\neg \exists y (y \in^{\mathcal{W}} x)$. 
We denote this element by $\emptyset^\mathcal{W}$.
If for all bounded formulas $\phi(x, \bar{y})$, all $\bar{a} \in \mathcal{W}$, and 
all ordinals $\alpha+2 < \beta$, there is some element $b \in (V_{\alpha+2}^\mathcal{W} \setminus  
V_{\alpha+1}^{\mathcal{W}}) \cup \{\emptyset^{\mathcal{W}}\}$ such that 
\begin{displaymath}
\mathcal{W} \models \forall x \in V_{\alpha+1}^\mathcal{W} \setminus V_\alpha^\mathcal{W} (\phi(x, \bar{a}) \leftrightarrow x \in^\mathcal{W} b),
\end{displaymath}
we say that $\mathcal{W}$ satisfies the Comprehension Axioms and denote it by $\mathcal{W} 
\models \textrm{CA}$.
It is easy to see that for any model $\mathcal{M}$ and any countable ordinal $\beta$, $S_\beta(\mathcal{M}) \models \textrm{CA}$ and if  $\mathcal{W}$ and $\mathcal{W}'$ are $\beta$-order-models such that $\mathcal{W} \models \textrm{CA}$ and $\mathcal{W} \equiv^B \mathcal{W}'$, then  $\mathcal{W}' \models \textrm{CA}$.
\item If $\mathcal{W}$ is a $\beta$-order-model for some countable ordinal $\beta$,  $\mathcal{W} \models \textrm{CA}$, and $\alpha+n+2<\beta$ for some ordinal $\alpha$ and some $n<\omega$, then $(V_{\alpha+1}^\mathcal{W} \setminus V_\alpha^\mathcal{W})^n \in V_{\alpha+n+2}^\mathcal{W} \setminus V_{\alpha+n+1}^\mathcal{W}$.
\end{enumerate}
\end{Remark}

Elementary equivalence of $\beta$-order-models $\mathcal{W}$ and $\mathcal{W}'$ implies the equivalence of  $(P_0^{\mathcal{W}}, R_0^{\mathcal{W}}, \ldots, R_n^{\mathcal{W}})$
and $(P_0^{\mathcal{W}'}, R_0^{\mathcal{W}'}, \ldots, R_n^{\mathcal{W}'})$ in higher order logic with the respective two versions of Henkin semantics encoded by $\mathcal{W}$ and $\mathcal{W}'$. If these structures are $L_{\kappa\omega}$-equivalent, the equivalence of the ground level structures extends even to infinitary higher order logic with the respective versions of Henkin semantics. Using the transfinite Ehrenfeucht-Fra\"iss\'e game $\mathop{\rm EF}_t^\kappa (\mathcal{W}^1, \mathcal{W}^2)$ (see e.g. \cite{HytTu} and \cite{vaa}) {\rm (}we define $\kappa^{+}$,$\kappa$ -trees and $\mathop{\rm EF}_t^\kappa$ as in \cite{HytTu}{\rm )} we get even stronger results.

\begin{teoreema}\label{lastbutone}
Let $\beta$ be a countable ordinal and let $\mathcal{W}$ be a $\beta$-order-model such that 
$\textrm{Th}(V_0^\mathcal{W}, R_0^\mathcal{W}, \ldots, R_n^\mathcal{W})$ is unstable. 
Let $\kappa$ be either a regular cardinal or a strong limit cardinal.
There are $\beta$-order-models $\mathcal{W}^1$ and $\mathcal{W}^2$ such that:
\begin{enumerate}
\item  $|W^1|=|W^2|=\kappa$.
\item $\mathcal{W} \equiv^B \mathcal{W}^1 \equiv^B \mathcal{W}^2$. In particular, if $\mathcal{W}$ satisfies CA, then so do $\mathcal{W}^1$ and $\mathcal{W}^2$.  If $\beta < \omega$, then $\mathcal{W} \equiv_{\omega\omega} \mathcal{W}^1$. 
\item  $\mathcal{W}^1 \equiv_{\infty\kappa} \mathcal{W}^2$. 
\item  
$(V_0^{\mathcal{W}^1}, R_0^{\mathcal{W}^1}, \ldots, R_n^{\mathcal{W}^1})\ncong (V_0^{\mathcal{W}^2}, R_0^{\mathcal{W}^2}, \ldots, R_n^{\mathcal{W}^2})$.
\end{enumerate}
\item If $\kappa^{<\kappa}=\kappa$, then for any $\kappa^{+}$,$\kappa$ -tree $t$ we can choose $\mathcal{W}^1$ and $\mathcal{W}^2$ so that in addition to (1)-(4), $\textrm{II} \uparrow \textrm{EF}_t^\kappa (\mathcal{W}^1, \mathcal{W}^2)$.
\end{teoreema}

Before proving the theorem we prove a crucial lemma:
 
\begin{lemma}\label{lemma2}
Suppose that $\beta$ is a countable ordinal and $\mathcal{W}$ is a $\beta$-order-model. Assume that there are sequences  $w_i \in V_0^{\mathcal{W}}$, for $i<\beth_{(2^\omega)^{+}}$, such that there exists a first-order formula $\phi$, for which
$(V_0^{\mathcal{W}}, R_0^{\mathcal{W}}, \ldots, R_n^{\mathcal{W}}) \models \phi(w_i, w_j)$ if and only if $i<j$.
If $\kappa > \omega$ is a regular cardinal or a strong limit cardinal, then there are $\beta$-order-models $\mathcal{W}^1$ and $\mathcal{W}^2$ such that:
\begin{enumerate}
\item $|V_0^{\mathcal{W}^1}|=|W^1|=|W^2|=|V_0^{\mathcal{W}^2}|$, 
\item
$\mathcal{W}^1 \equiv_{\omega\omega}  \mathcal{W}^2 \equiv_{\omega\omega} \mathcal{W}$, 
\item $\mathcal{W}^1 \equiv_{\infty\kappa} \mathcal{W}^2$,
\item $(V_0^{\mathcal{W}^1}, R_0^{\mathcal{W}^1}, \ldots, R_n^{\mathcal{W}^1}) \ncong (V_0^{\mathcal{W}^2}, R_0^{\mathcal{W}^1}, \ldots, R_n^{\mathcal{W}^2})$.  
\end{enumerate}
If  $\kappa^{<\kappa}=\kappa$, then for any $\kappa^{+}$,$\kappa$ -tree $t$ we can choose $\mathcal{W}^1$ and $\mathcal{W}^2$ so that in addition to (1)-(4), the second player (i.e. the ``equivalence"-player) has a winning strategy in the transfinite Ehrenfeucht-Fra\"iss\'e game $\mathop{\rm EF}_t^\kappa (\mathcal{W}^1, \mathcal{W}^2)$ {\rm (}we define $\kappa^{+}$,$\kappa$ -trees and $\mathop{\rm EF}_t^\kappa$ as in \cite{HytTu}{\rm )} .\end{lemma}

\begin{proof}
We add to $L$ a new constant symbol $c$, and let $c^{\mathcal{W}}=\textrm{Pr}^{\mbox{len}(w_0)}_0(w_0)$.
After this we add Skolem functions and denote by $\mathcal{W}^\textrm{Sk}$ the structure obtained.
Let $(I,<)$ be a linear order.
It follows from Erd\"os-Rado Theorem that there is a template $\Phi$ such that if $\{a_i \, | \, i \in I\}$ is the skeleton of $\textrm{EM}^1(I, \Phi)$, then for any $i_1, \ldots, i_n \in I$, there are $j_1, \ldots, j_n \in V_0^{\mathcal{W}}$ so that the mapping $\pi: \textrm{SH}(a_{i_1}, \ldots, a_{i_n}) \to \mathcal{W}^\textrm{Sk}$, $a_{i_k} \mapsto w_{j_k}$,   for $1 \leq k \leq n$, is an elementary embedding.

Thus, $\textrm{EM}(I, \Phi)$ is a $\beta$-order-model.
(If the conditions from Definition \ref{betamod} would not hold in $\textrm{EM}(I, \Phi)$, this would be witnessed already in some finite $A \subseteq \textrm{EM}(I, \Phi)$. This, of course, is impossible as we have the elementary embeddings.)
Also, it follows from the definition of an Ehrenfeucht-Mostowski model, that if $I$ and $J$ are linear orders such that $I \equiv J$ $(L_{\infty \kappa})$, then $\textrm{EM}(I, \Phi) \equiv \textrm{EM}(J, \Phi)$ $(L_{\infty \kappa})$.

Now we close the Skolem functions with respect to projections and compositions as follows.
If $f$ is an $n$-place function, then for each $m \in \omega$ and each tuple $(i_1, \ldots, i_m)$, where $1 \leq i_k \leq m$ for $1 \leq k \leq m$, we add a new $m$-place function symbol $g$ and interpret it as follows.
If $\bar{a} \in \textrm{EM}(I, \Phi)^m$, then
\begin{displaymath}
g(\bar{a})=f(\textrm{Pr}_{i_1}^m(\bar{a}), \ldots, \textrm{Pr}_{i_n}^m(\bar{a})).
\end{displaymath}
After this we add for each $n$-place function $f$ and for each $n$-tuple of $m$-place functions $g_1, \ldots, g_n$, a new function symbol $h$ that is interpreted as follows.
If $\bar{a} \in \textrm{EM}(I, \Phi)^m$, then
\begin{displaymath}
h(\bar{a})=f(g_1(\bar{a}), \ldots, g_n(\bar{a})).
\end{displaymath}
We repeat the process $\omega$ many times and denote by $F^{*}$ the collection of functions obtained this way.

For each $f \in F^{*}$, we define a function $f^0$ so that for each $\bar{a} \in \textrm{EM}(I, \Phi)$,
\begin{displaymath}
f^0(\bar{a})=\left\{ \begin{array}{ll}
f(\bar{a}) & \textrm{if $\bar{a} \in V_0^{\textrm{EM}(I, \Phi)}$}
\\ c^{\textrm{EM}(I, \Phi)} & \textrm{otherwise}\\  
\end{array} \right.
\end{displaymath} 
Let $F=\{f^0 \, | \, f \in F^{*}\}$.
Denote by $\textrm{EM}^2(I, \Phi)$ the model we get by interpreting in $\textrm{EM}^1(I,\Phi)$ the new Skolem-functions. 
The new functions are definable in $\textrm{EM}^1(I,\Phi)$, and thus the skeleton $\{a_i \, | \, i \in I \}$ is a set of order indiscernibles in $\textrm{EM}^2 (I,\Phi)$.
Moreover, $V_0^{\textrm{EM}(I,\Phi)}=\textrm{SH}^F(\{a_i \, | \, i \in I \})$.

As in \cite{Sh2} we can construct $L_{\infty \kappa}$-equivalent linear orders $I$ and $J$ such that the models  $\textrm{EM}(I, \Phi)$ and $\textrm{EM}(J, \Phi)$ are non-isomorphic (to be exact, in \cite{Sh2}, instead of linear orders, trees with $\omega+1$ levels are constructed, but as in e.g. \cite{FrHyKu}, proof of Theorem 80, Claim 4, trees can be coded as linear orders).
 
If $\kappa^{<\kappa}=\kappa$, we choose linear orders $I$ and $J$ as in \cite{HytTu} to obtain $\textrm{EM}(I, \Phi)$ and $\textrm{EM}(J, \Phi)$ as wanted.
\end{proof}

In order to apply the lemma for the proof of Theorem~\ref{lastbutone} we now introduce an ultraproduct construction:

\begin{definition}
Let $I$ be a set and $D$ an ultrafilter over $I$.
Let $\beta$ be a countable ordinal, let $\mathcal{W}_i$, $i \in I$, be $\beta$-order-models, and let $\prod_I \mathcal{W}_i/D$ be the ultaproduct of the models $\mathcal{W}_i$ modulo $D$.
We let $\prod^B_I \mathcal{W}_i / D \subseteq \prod_I \mathcal{W}_i/D$ be such that $f/D \in \prod^B_I \mathcal{W}_i / D$  if and only if either
\begin{displaymath}
\{i \in I \, | \, f(i) \in V_{\alpha+1}^{\mathcal{W}_i} \setminus V_{\alpha}^{\mathcal{W}_i} \} \in D
\end{displaymath}
 for some ordinal $\alpha+1 < \beta$, or
 \begin{displaymath}
 \{i \in I \, | \, f(i) =\underline{\alpha}^{\mathcal{W}_i}  \} \in D
\end{displaymath}
for some ordinal $\alpha < \beta$.
We call $\prod^B_I \mathcal{W}_i / D$ \emph{the bounded ultraproduct of the models $\mathcal{W}_i$ modulo $D$}. 
If $\mathcal{W}_i=\mathcal{W}$ for every $i \in I$, then we write $\prod^B_I \mathcal{W}_i/D=\prod^B_I \mathcal{W}/D$ and call it \emph{the bounded ultrapower of the model $\mathcal{W}$ modulo $D$}.
\end{definition}

\begin{Remark}
Let $\mathcal{W}_i$, $i \in I$ be $\beta$-order-models for some countable ordinal $\beta$.
\begin{enumerate}[(i)]
\item 
$\prod^B_I \mathcal{W}_i / D \preccurlyeq^B \prod_I \mathcal{W}_i / D$.

To show this, we have to prove that whenever $\phi(x_1, \ldots, x_n)$ is a bounded formula, and $f_1/D, \ldots, f_n/D \in \prod^B_I \mathcal{W}_i / D$, then $\prod^B_I \mathcal{W} / D \models \phi(f_1/D, \ldots, f_n/D)$ if and only if $\prod_I \mathcal{W} / D \models \phi(f_1/D, \ldots, f_n/D)$. 
The claim clearly holds if $\phi$ is an atomic formula or a Boolean combination of atomic formulas.
Suppose now the claim holds for $\psi(x_1, \ldots, x_n)$ and $\phi(x_1, \ldots, x_n)=(\exists y \in V_{\alpha+1}\setminus V_\alpha) \psi(y, x_1, \ldots, x_n)$ for some ordinal $\alpha+1 < \beta$.
The direction from left to right is clear.
For right to left, suppose $\prod_I \mathcal{W}_i / D \models \phi(f_1/D, \ldots, f_n/D)$.
Thus, there is some $g/D \in V_{\alpha+1}^{\prod_I \mathcal{W}_i / D} \setminus V_\alpha^{\prod_I \mathcal{W}_i / D}$ such that $\prod_I \mathcal{W}_i / D \models \psi(g/D, f_1/D, \ldots, f_n/D)$.
By \L o\'s's Theorem, $\{i \in I \, | \, g(i) \in V_{\alpha+1}^{\mathcal{W}_i} \setminus V_\alpha^{\mathcal{W}_i} \} \in D$.
Thus, $g/D \in \prod^B_I \mathcal{W}_i / D$, and the claim follows.
\item By (i), \L o\'s's Theorem holds for all bounded sentences in the case of bounded ultraproducts, and it follows that for any $\beta$-order-model $\mathcal{W}$,  $\mathcal{W} \equiv^B\prod^B_I \mathcal{W} / D$. 
\item If $\beta=n<\omega$, then $\prod^B_I \mathcal{W}_i / D = \prod_I \mathcal{W}_i/D$.
\item $\prod^B_I \mathcal{W}_i / D$ is a $\beta$-order-model.
\end{enumerate}
\end{Remark}

Now we are ready to prove Theorem \ref{lastbutone}:

\begin{proof}
Let $\lambda \geq \beth_{(2^\omega)^{+}}$ and let $D$ be a $\lambda$-regular ultrafilter over $\lambda$.
Let $\mathcal{W}^{*}=\prod^B_\lambda \mathcal{W} / D$. 
It clearly holds that 
\begin{displaymath}
(V_0^{\mathcal{W}^{*}}, R_0, \ldots, R_n) \cong \prod_\lambda (V_0^{\mathcal{W}}, R_0, \ldots, R_n)/D.
\end{displaymath}
By \cite{CK}, Theorem 4.3.12, $\prod_\lambda (V_0^{\mathcal{W}}, R_0, \ldots, R_n)/D$ is $\lambda^{+}$-universal.
Thus, as we assumed that $\textrm{Th}(V_0^\mathcal{W}, R_0, \ldots, R_n)$ is unstable, there are elements $w_i \in V_0^{\mathcal{W}^{*}}$, $i \leq \beth_{(2^\omega)^{+}}$, such that $w_i \neq w_j$ for $i \neq j$, and a first-order formula $\phi$ such that 
\begin{displaymath}
(V_0^{\mathcal{W}^{*}}, R_0, \ldots, R_n) \models \phi(w_i, w_j) \iff i<j.
\end{displaymath}
Hence we can apply Lemma \ref{lemma2} to the $\beta$-order-model $\mathcal{W}^{*}$. 

If $\beta$ is finite, then $\mathcal{W}^{*}=\prod_\lambda \mathcal{W} / D$, and $\mathcal{W} \equiv_{\omega \omega} \mathcal{W}^{*} \equiv_{\omega \omega} \mathcal{W}^1$.
\end{proof}

The models $\mathcal{W}^1$ and $\mathcal{W}^2$ of the above theorem have a different higher order part, i.e. predicates $P_\gamma$. Thus they represent different versions of higher order logic. Can we get the same result with more or less the {\em same} higher order part. This would be in line with the situation with the full models, where the higher order part is fixed to be built from the real power-sets. We prove a result in this direction but to obtain the desired identity of the higher order components we have to make cardinality arithmetic assumptions.

We define a version of the concept of a $\beta$-order-model appropriate for the situation: 

\begin{definition}
Let $L'=\{R_0, \ldots, R_n \}$ be a relational vocabulary, and let $\ma$ and $\mb$
be $L'$-models such that $\textrm{dom}(\ma) \cap \textrm{dom}(\mb) = \emptyset$.
Let  $L^{*}=\{R_0, \ldots, R_n, R_0', \ldots, R_n', P_{\ma}, P_{\mb} \}$ and let $\mathcal{C}$ be an $L^{*}$-model such that
\begin{itemize}
\item $\textrm{dom}(\mathcal{C})=\textrm{dom}(\ma) \cup \textrm{dom}(\mb)$, 
\item ${P_\ma}^{\mathcal{C}}=\textrm{dom}(\ma)$, ${P_{\mb}}^{\mathcal{C}}=\textrm{dom}(\mb)$, 
\item  $R_i^{\mathcal{C}}=R_i^{\ma}$, $R_i'^{\mathcal{C}}=R_i^{\mb}$, for $0 \leq i \leq n$.
\end{itemize} 
If  $\mathcal{W}$ is a $\beta$-order-model in the vocabulary $L=L^{*} \cup \{P_\beta, <, V, \in\}\cup\{\underline{\alpha} \, | \, \alpha<\beta\}$, and  $(V_0^\mathcal{W}, R_0, \ldots, R_n, R_1', \ldots, R_n', P_\ma, P_\mb) = \mathcal{C}$, we say $\mathcal{W}$ is a \emph{$\beta$-order-model over the models $\ma$ and $\mb$}.
We call $S_\beta(\mathcal{C})$ \emph{the full $\beta$-order-model over $\ma$ and $\mb$} and denote it by $S_\beta(\ma, \mb)$. 

If $\phi$ is an $L$-formula not containing the symbols $R_0', \ldots, R_n'$ and $P_\mb$, we denote by $\phi^{*}$ the formula obtained from $\phi$ by replacing the symbol $R_i$ by the symbol $R_i'$ for $0 \leq i \leq n$, and the symbol $P_\ma$ by the symbol $P_\mb$.
 \end{definition}

Note that in a $\beta$-order-model over the models $\ma$ and $\mb$ the domains of both $\ma$ and $\mb$ constitute the ground level $V_0$. So whatever the formal sets there are on the higher levels $P_i$ their common properties are properties of subsets of the domain of $\ma$ as much as of $\mb$. In this way the below theorem is closer to the results of Section~\ref{fot} than Theorem~\ref{lastbutone} is. 

\begin{teoreema}
Suppose $T$ is a countable unstable theory in the language $L'=\{R_0, \ldots, R_n\}$, and $\kappa$ and $\lambda$ are cardinals such that $\kappa=\lambda^{+}=2^\lambda$, $\lambda^{<\lambda}=\lambda > \omega$.
Then there are non-isomorphic models $\ma, \mb \models T$ and an $\omega$-order-model $\mathcal{W} \models \textrm{CA}$ over $\ma$ and $\mb$ such that for each $L_{\kappa \omega}$-sentence $\phi$ not containing the symbols $R_0', \ldots, R_n'$ and $P_\mb$, it holds that
 \begin{displaymath}
 \mathcal{W} \models \phi \iff \mathcal{W} \models \phi^{*}.
\end{displaymath} 
\end{teoreema}

\begin{proof}
Let $M^{*} \preccurlyeq H(\chi)$ for some cardinal $\chi$ that is large enough, and let $|M^{*}|=\kappa$, ${M^{*}}^{<\kappa} \subseteq M^{*}$, and suppose $M^{*}$ is such that it contains everything needed later in the proof, especially every $L_{\kappa \omega}$ -formula.
Denote by $M$ the Mostowski collapse of $M^{*}$. 
As in Theorem \ref{pakotus}, let  $\mathbb{P}=\{f: \alpha \rightarrow \{0,1\} \, | \, \alpha < \kappa\}$, ordered by inclusion.
Let $G$ be a $\mathbb{P}$-generic filter over $M$, and let $S_G=(\bigcup G)^{-1}(1) \cap S^\kappa_\lambda$.
Using similar notation as in Lemma \ref{lemma1}, let $\ma_G=\textrm{EM}(L(T(S_G)), \Phi_T)$, and $\mb_G=\textrm{EM}(L(T(S^\kappa_\lambda \setminus S_G)), \Phi_T)$.
These models can be formed either in the extension $M[G]$ or in $V$, and the result is the same.
We may suppose without loss of generality, that $\ma_G \cap \mb_G=\emptyset$.

Let $\mathcal{W}_G=S_\omega^{M[G]}(\ma, \mb)$. 
Suppose $\phi$ is an $L_{\kappa \omega}$-sentence such that $\mathcal{W}_G \models \phi$. 
It suffices to show that $\mathcal{W}_G \models \phi^{*}$ in $M[G]$.
Clearly $\mathcal{W}_G \models \phi$ in $M[G]$, and thus there is some $p \in G$ such that in $M$
 \begin{displaymath}
 p \Vdash \dot{\mathcal{W}}_G \models \check{\phi}.
 \end{displaymath}
Let $\textrm{dom }p=\gamma$, and denote
\begin{displaymath}
G^*_\gamma=\{f^*_\gamma \, | \, f \in G \},
\end{displaymath}
where 
\begin{displaymath}
f^*_\gamma(\alpha)=\left\{ \begin{array}{ll}
f(\alpha) & \textrm{if $\alpha < \gamma$}\\ 1-f(\alpha) & \textrm{otherwise.}\\  
\end{array} \right.
\end{displaymath}
Then $\mathcal{W}_{G^{*}} \models \phi$ in $M[G^{*}]=M[G]$. 
Also, as $\ma_{G} \bigtriangleup \mb_{G^{*}}$ and $\mb_{G} \bigtriangleup \ma_{G^{*}}$ are non-stationary, there are isomorphisms $f: \ma_{G} \to \mb_{G^{*}}$ and $g: \mb_{G} \to \ma_{G^{*}}$ in $M[G]$.
There is in $M[G]$ (and hence in $V$) a function $F: \mathcal{W}_G \to  \mathcal{W}_{G^{*}}$ such that $F \upharpoonright \ma_G = f$, $F \upharpoonright \mb_G = g$, and $F(x)=\{F(y) \, | \, y \in^{\mathcal{W}_G} x \}$ if $x \in \mathcal{W}_G \setminus (\ma_G \cup \mb_G)$.
Since $\mathcal{W}_{G^{*}}=S_\omega^{M[G]}(\ma_{G^{*}}, \mb_{G^{*}})$, $F$ is an "isomorphism" in the sense that if $\psi(\bar{x})$ is an $L^{*}$-formula not containing the symbols $R_0', \ldots, R_n'$, $P_\mb$, and $\bar{a} \in \ma_G$, $\bar{b} \in \mb_G$, then
\begin{displaymath}
\ma_G \models \psi(\bar{a}) \iff \mb_{G^{*}} \models \psi^{*}(F(\bar{a})),
\end{displaymath}
and
\begin{displaymath}
\mb_G \models \psi^{*}(\bar{b}) \iff \ma_{G^{*}} \models \psi(F(\bar{b})).
\end{displaymath}
Thus, $\mathcal{W}_G \models \phi^{*}$ in $M[G]$.
The other direction is symmetric.
\end{proof}

\bibliographystyle{plain}

\bibliography{HytKanVaa}
\medskip

Department of Mathematics

University of Helsinki
\medskip

Institute of Logic, Language and Computation

University of Amsterdam

\end{document}